\documentclass[11pt]{article}
\usepackage{geometry}
\usepackage{amsmath}
\usepackage{amssymb}
\usepackage{amsthm}
\usepackage{amsfonts}
\usepackage[dvipsnames]{xcolor}
\usepackage{tikz}
\usepackage{mathtools}
\usepackage{pifont}
\usepackage{tabularx}
\usepackage{comment}
\usepackage{graphicx}

\usepackage[bookmarks,bookmarksnumbered,pdfstartview=FitBH,backref=page]{hyperref}
\hypersetup{
	colorlinks=true, 
	linkcolor=black,
	citecolor= black,
        urlcolor=black,
	linktoc= all }

\DeclareMathOperator{\id}{id}
\DeclareMathOperator{\Isom}{Isom}

\DeclareMathOperator{\comp}{comp}

\newcommand{\rkCB}{\mathrm{rk}_\mathrm{CB}}

\newtheorem{theorem}{Theorem}[section]
\newtheorem{lemma}[theorem]{Lemma}
\newtheorem{proposition}[theorem]{Proposition}

\newtheorem{claim}{Claim}

\theoremstyle{definition}
\newtheorem{definition}[theorem]{Definition}

\theoremstyle{remark}
\newtheorem{remark}[theorem]{Remark}

\let\OLDthebibliography\thebibliography
\renewcommand\thebibliography[1]{
  \OLDthebibliography{#1}
  \setlength{\parskip}{0pt}
  \setlength{\itemsep}{2pt plus 0.3ex}
}

\title{Uncountably many homogeneous real trees \\ with the same valence}
\author{Pénélope Azuelos}
\date{November 5, 2025}

\begin{document}

\maketitle

\begin{abstract}
For any cardinal $\kappa \geq 2$, there is a unique complete real tree whose points all have valence $\kappa$. In this note, we show that, when $\kappa \geq 3$, it is necessary to assume completeness. More precisely, we show that there exist uncountably many homogeneous incomplete real trees whose points all have valence $\kappa$.
\end{abstract}

\section{Introduction}

Real trees are geodesic metric spaces with the property that any pair of points is connected by a unique arc. One of the ways in which these spaces arise is as asymptotic cones of hyperbolic groups. In fact, the asymptotic cone of a non-elementary hyperbolic group is always isometric to the same real tree: the universal real tree with continuum valence \cite{DyubinaExplicit01}. The \textit{valence} of a real tree $T$ at a point $x$ is the cardinality of the set of connected components of $T - \{x\}$ (the \textit{directions at $x$}). It was proven in \cite{MayerUniversal92,NikielTopologies89} that, for any cardinal $\kappa \geq 2$, there is a unique complete real tree $T_\kappa$, called the universal real tree of valence $\kappa$, such that every point of $T_\kappa$ has valence $\kappa$. Explicit constructions of these universal real trees where given in \cite{DyubinaExplicit01}.

While incomplete real trees do not arise as asymptotic cones directly, they do come up in this context. The asymptotic cone of a relatively hyperbolic group is a universal tree-graded space whose pieces are asymptotic cones of peripheral subgroups \cite{OsinUNIVERSAL11,SistoTreegraded13}. This tree-graded space admits a canonical projection to a real tree \cite[Section~2.3]{DrutuGroups08}, which is homogeneous and, provided the parabolic subgroups are infinite, is incomplete. Furthermore, the asymptotic cones of mapping class groups and other hierarchically hyperbolic groups admit canonical median preserving bi-Lipschitz embeddings into $\ell^1$ products of real trees \cite{BehrstockMedian11,Casals-RuizReal24}. Some of the images of these asymptotic cones under the natural projections to the factors are homogeneous incomplete real trees. This class of real trees also arises in work of Chiswell--M\"{u}ller, who proved that any free non-transitive group action on a real tree can be embedded into a free transitive action on an incomplete real tree \cite{ChiswellEmbedding10}. Lastly, Berestovski\u{\i} \cite{BerestovskiuiUryson19} noticed that a space first defined by Urysohn \cite{UrysohnBeispiel27}, and  which turns out to be a homogeneous incomplete real tree, was independently studied in \cite{BerestovskiuiQuasicones89,PolterovichAsymptotic97}. 

We will show that, unlike in the complete setting, valence alone is insufficient to distinguish between homogeneous real trees. More precisely, we will show the following:

\begin{theorem}
    Let $\kappa \geq 3$ be a cardinal. There exists a family $\{T_\kappa^{[\alpha]} : 1 \leq \alpha < \omega_1\}$ of incomplete homogenous real trees with valence $\kappa$ and a fixed isometric embedding $\psi_{\alpha,\beta}: T_\kappa^{[\alpha]} \hookrightarrow T_\kappa^{[\beta]}$ for all $\alpha < \beta$ such that:
    \begin{enumerate}
        \item for any countable ordinals $\alpha, \beta$, if $\alpha \neq \beta$ then $T_\kappa^{[\alpha]}$ is not isometric to $T_\kappa^{[\beta]}$;
        \item if $\alpha < \beta < \gamma$, then $\psi_{\alpha,\gamma} = \psi_{\beta,\gamma} \circ \psi_{\alpha,\beta}$;
        \item the direct union $\cup_{\alpha < \omega_1} T_\alpha$, taken with respect to $\{\psi_{\alpha,\beta} : 1 \leq \alpha < \beta < \omega_1\}$, is isometric to the universal real tree with valence $\kappa$.
    \end{enumerate}
\end{theorem}

\begin{proof}
	We will use the description of the universal real tree $T_\kappa$ from \cite{DyubinaExplicit01} to define a filtration $T_\kappa = \cup_{\alpha < \omega_1} T_\kappa^{[\alpha]}$ (Definition~\ref{def: tka}). In Lemma~\ref{lem: incomplete r-tree}, we prove that each $T_\kappa^{[\alpha]}$ is an $\mathbb{R}$-tree with valence $\kappa$ and, in Lemma~\ref{lem: 2-transitive}, we show that $T_\kappa^{[\alpha]}$ is homogeneous. Letting $\psi_{\alpha,\beta}$ be the inclusion maps, this proves Items 2 and 3. We prove Item 1 and show that the $T_\kappa^{[\alpha]}$'s are incomplete in Proposition~\ref{lem: no isometry}.
\end{proof}

\begin{remark}
    \begin{enumerate}
        \item The real tree discussed in \cite{BerestovskiuiUryson19} is $T_{2^{\aleph_0}}^{[1]}$.
        \item If one only considers real trees which admit a free transitive action, then valence does distinguish between real trees of finite valence \cite[Theorem~C]{AzuelosGuide25}. This is no longer true for continuum valence trees, and I do not know if it holds for intermediate cardinals.
    \end{enumerate}
\end{remark}

\section{Cantor--Bendixson rank}

We will make use of a topological invariant called the Cantor--Bendixson rank to distinguish between different real trees with the same valence. Let us recall its definition.

Let $Y$ be a Polish space (i.e. separable and completely metrisable). The Cantor--Bendixson theorem (see e.g. \cite{KechrisClassical95}) states that there is a unique decomposition of $Y$ as a disjoint union $\mathcal{K}(Y) \sqcup C$, where $\mathcal{K}(Y)$ is perfect (i.e. closed with no isolated points) and $C$ is countable. Any non-empty perfect Polish space contains a Cantor set, so the perfect subspace $\mathcal{K}(Y)$ is empty if and only if $Y$ is countable. 

The Cantor--Bendixson derivatives of $Y$ are defined by transfinite recursion. Set $Y^{(0)} \coloneqq Y$. If $\alpha$ is an ordinal for which $Y^{(\alpha)}$ is defined, then $Y^{(\alpha+1)}$ is the set of non-isolated points of $Y^{(\alpha)}$. If $\beta$ is a limit ordinal such that $Y^{(\alpha)}$ is defined for all $\alpha < \beta$, then $Y^{(\beta)} \coloneqq \cap_{\alpha < \beta} Y^{(\alpha)}$. The Cantor--Bendixson theorem implies that there is a countable ordinal $\alpha$ such that $Y^{(\alpha + 1)} = Y^{(\alpha)}$ (i.e. $Y^{(\alpha)} = \mathcal{K}(Y)$). The minimal such $\alpha$ is called the  \textit{Cantor--Bendixson rank} of $Y$ and is denoted $\rkCB(Y)$.

\begin{remark} \label{rem: CB-rank succesor}
    If $Y$ is a countable compact metrisable space then its Cantor--Bendixson rank must be a successor ordinal. 
\end{remark}

\section{Finding the real trees}

Let $\kappa \geq 3$ be a cardinal. We first recall the construction of the complete universal real tree with valence $\kappa$ from \cite{DyubinaExplicit01}:

\begin{definition}[Labels]
    Let $C_\kappa$ be a set such that $ |C_\kappa| = \kappa$ if $\kappa$ is infinite and $|C_\kappa| = \kappa - 1$ if $\kappa$ is finite. Fix an element $0 \in C_\kappa$. 
\end{definition}

\begin{definition}
    Let $T_\kappa$ be the set of functions $f: (-\infty, \rho_f) \rightarrow C_\kappa$, where $\rho_f \in \mathbb{R}$, such that:
    \begin{enumerate}
        \item There exists $s \leq \rho_f$ such that $f(t) = 0$ for all $t < s$.
        \item $f$ is \textit{piecewise constant from the right}: i.e. for all $t \in (-\infty, \rho_f)$, there exists $\varepsilon > 0$ such that $f|_{[t, t+\varepsilon]}$ is constant.
    \end{enumerate}
    Given $f \in T_\kappa$, let $\tau_f$ be the maximum of all $s \leq \rho_f$ such that $f(t) = 0$ for all $t < s$. 
\end{definition}

\begin{definition}[Partial order]
	Define an order $\preceq$ on $T_\kappa$ where $f \preceq g$ if $f = g|_{(-\infty, \rho_f)}$. We will write $f \prec g$ if in addition $f \neq g$.
	
	Given $f,g \in T_\kappa$, let $s \coloneqq \sup\{t \leq \min\{\rho_f, \rho_g\} : f|_{(-\infty,t)} = g|_{(-\infty, t)}\}$ and define $f \wedge g: (-\infty, s) \rightarrow C_\kappa$ by $f \wedge g(t) = f(t) = g(t)$ for all $t < s$.
\end{definition}

\begin{remark}
    The relation $\preceq$ is reflexive, transitive and antisymmetric. Moreover, if $f, g \in T_\kappa$, then $f \wedge g$ is the greatest lower bound for $\{f,g\}$. 
\end{remark}

\begin{definition}[Metric]
    Define $d:T_\kappa \times T_\kappa \rightarrow \mathbb{R}$ by 
    \[
        d(f,g) \coloneqq \rho_f + \rho_g - 2\rho_{f \wedge g}
    \]
    for all $f,g \in T_\kappa$.
\end{definition}

By \cite[Theorem~1.1.3]{DyubinaExplicit01} $(T_\kappa, d)$ is the universal $\mathbb{R}$-tree with valence $\kappa$.

\begin{definition}[Complexity] \label{def: tka}
    Given $f \in T_\kappa$, let 
    \[
        P_f \coloneqq \overline{\{t < \rho_f : f|_{[t-\varepsilon, t]} \text{ is not constant for any } \varepsilon > 0\}}.
    \]
    The \textit{complexity of $f$}, denoted $\comp(f)$, is the Cantor--Bendixson rank of $P_f$.

    For each countable ordinal $\alpha$, define
    \[
        T_\kappa^{[\alpha]} \coloneqq \{f \in T_\kappa : \comp(f) \leq \alpha\}. 
    \]
\end{definition}

\begin{remark} 
Given $f \in T_\kappa$, the fact that $f$ is piecewise constant from the right ensures that the set $P_f$ is well-ordered with respect to the usual order on $\mathbb{R}$. This implies that $P_f$ is countable, so its perfect kernel is empty. Moreover, $P_f$ is closed and contained in the interval $[\tau_f, \rho_f]$, so it is compact. Thus, by Remark~\ref{rem: CB-rank succesor}, the Cantor--Bendixson rank of $P_f$ is a successor.
\end{remark}

\begin{lemma} \label{lem: incomplete r-tree}
    If $\alpha \geq 1$ is a countable ordinal, then $(T_\kappa^{[\alpha]}, d)$ is an $\mathbb{R}$-tree with valence~$\kappa$.
\end{lemma}
\begin{proof}
    If $f,g \in T_\kappa$, $f \preceq g$ and $\comp(g) \leq \alpha$, then $\comp(f) \leq \alpha$. Thus $T_\kappa^{[\alpha]}$ is downward closed, which implies that it is connected and therefore an $\mathbb{R}$-tree. Let $f \in T_\kappa^{[\alpha]}$. One direction of $T_\kappa$ at $f$ consists of all elements $g \in T_\kappa$ such that $g \wedge f \prec f$. In particular, this direction contains all elements of the form $f|_{(-\infty,s)}$ for some $s < \rho_f$, so its intersection with $T_\kappa^{[\alpha]}$ is non-empty. The remaining directions of $T_\kappa$ at $f$ are of the form $D_x = \{g \in T_\kappa : f \prec g$ and $g(\rho_f) = x\}$ for some $x \in C_\kappa$. Define $g: (- \infty, \rho_f + 1) \rightarrow \infty$ by $g(t) = f(t)$ if $t < \rho_f$ and $g(t) = x$ otherwise. Then $P_g \subseteq P_f \cup \{\rho_f\}$, so $g \in T_\kappa^{[\alpha]} \cap D_x$. Therefore the valence of $T_\kappa^{[\alpha]}$ at $f$ is $\kappa$.
\end{proof}

\begin{definition}
    For each $\ell \in \mathbb{R}$, let $c_\ell: (-\infty,\ell) \rightarrow C_\kappa$ be the constant map on $0$. Let 
    \[
        L_0 \coloneqq T_\kappa^{[0]} = \{c_\ell : \ell \in \mathbb{R}\}
    \]
    and note that the function $L_0 \rightarrow \mathbb{R}$ which maps each $c_\ell$ to $\ell$ is an isometry. Let $G^\alpha_{L_0}$ be the stabiliser of $L_0$ in $\Isom(T_\kappa^{[\alpha]})$.
\end{definition}

\begin{lemma} \label{lem: 2-transitive}
    Fix a countable ordinal $\alpha \geq 1$. Given $a_1, a_2, b_1, b_2 \in T_\kappa^{[\alpha]}$ such that $d(a_1, a_2) = d(b_1, b_2)$, there is an isometry $\psi \in \Isom(T_\kappa^{[\alpha]})$ such that $\psi(a_1) = b_1$ and $\psi(a_2) = b_2$. In particular, $T_\kappa^{[\alpha]}$ is homogeneous.
\end{lemma}
\begin{proof}
    \begin{claim} \label{claim 1}
        Given $a \in T_\kappa^{[\alpha]}$, there is an isometry $\varphi \in \Isom(T_\kappa^{[\alpha]})$ such that $\varphi(a) = c_{\rho_a}$ and $\varphi(c_r) = c_r$ for all $r \leq \tau_a$.
    \end{claim}
    \begin{proof}
    \renewcommand{\qedsymbol}{$\blacksquare$}
        If $a = c_{\rho_a}$ then the claim holds with $\varphi = \id$, so suppose that $a \neq c_{\rho_a}$. Define $\varphi: T_\kappa^{[\alpha]} \rightarrow T_\kappa^{[\alpha]}$ as follows. If $b \in T_\kappa^{[\alpha]}$ is such that $b \wedge a \wedge c_{\rho_a} \preceq c_{\tau_a}$, then $\varphi(b) \coloneqq b$. Otherwise, either:
        \begin{itemize}
            \item[i.] $b \wedge a \succ c_{\tau_a}$ and $b \wedge c_{\rho_a} = c_{\tau_a}$, or
            \item[ii.] $b \wedge a = c_{\tau_a}$ and $b \wedge c_{\rho_a} \succ c_{\tau_a}$.
        \end{itemize}
        Let $\sigma \coloneqq \rho_{b \wedge a}$ if i holds, and $\sigma \coloneqq \rho_{b \wedge c_{\rho_a}}$ if ii holds.
        Let $\varphi(b)(t) \coloneqq  b(t)$ if $t < \tau_a$ or $t \geq \sigma$. If $\tau_a \leq t < \sigma$ and i holds, then let $\varphi(b)(t) \coloneqq 0$. If $\tau_a \leq t < \sigma$ and ii holds, then let $\varphi(b)(t) \coloneqq a(t)$. It is straightforward to check that $\varphi(b) \in T_\kappa^{[\alpha]}$.
        Then in particular $\varphi(a) = c_{\rho_a}$ and $\varphi(c_{\rho_a}) = a$.
        
        Let us show that $\varphi$ is an isometry. First observe that $\varphi$ is an involution, so in particular it is a bijection. Let $b, d \in T_\kappa^{[\alpha]}$. Then $\rho_{\varphi(b)} = \rho_b$ and, if $b \preceq d$, then $\varphi(b) \preceq \varphi(d)$. It follows that $\varphi(b \wedge d) = \varphi(b) \wedge \varphi(d)$, so $\rho_{\varphi(b) \wedge \varphi(d)} = \rho_{b \wedge d}$.
        Thus
        \[
            d(\varphi(b), \varphi(d)) 
            = \rho_{\varphi(b)} + \rho_{\varphi(d)} - 2\rho_{\varphi(b) \wedge \varphi(d)} 
            = \rho_b + \rho_d - 2\rho_{b \wedge d}
            = d(b,d).
            \qedhere
        \]
    \end{proof}
    
    \begin{claim} \label{Claim 2}
        The homomorphism $G^\alpha_{L_0} \rightarrow \Isom(L_0)$ induced from the action of $\Isom(T_\kappa^{[\alpha]})$ is surjective.
    \end{claim}
    \begin{proof}
    \renewcommand{\qedsymbol}{$\blacksquare$}
        Let $r > 0$ and define $\varphi_r \in \Isom(T_\kappa^{[\alpha]})$ as follows. Given $a \in T_\kappa^{[\alpha]}$, let $\varphi(a):(-\infty, \rho_a + r) \rightarrow C_\kappa$ be defined by:
        \[ 
        	\varphi(a)(t) \coloneqq
        	\begin{cases}
        		0 \quad &\text{if } t < \tau_a + r; \\
        		a(t-r) &\text{otherwise.}
        	\end{cases}
        \]
        Then $\varphi_r \in G^\alpha_{L_0}$ and $\varphi_r$ acts as a translation of length $r$ on $L_0$.
        To define an element $\varphi \in G^\alpha_{L_0}$ which acts as a reflection on $L_0$, let $\varphi(a): (-\infty, \rho_a - 2\tau_a) \rightarrow C_\kappa$  be the map defined by $\varphi(a)(t) = 0$ if $t < -\tau_a$ and $\phi(a)(t) = a(t - 2\tau_a)$ otherwise, for all $a \in T_\kappa^{[\alpha]}$.
    \end{proof}

    Let $r \coloneqq d(a_1, a_2)$. By Claim~\ref{claim 1}, there exists $\varphi_1 \in \Isom(T_\kappa^{[\alpha]})$ such that $\varphi_1(a_1) = c_{\rho_{a_1}}$ and, by Claim~\ref{Claim 2}, there exists $\varphi_2 \in \Isom(T_\kappa^{[\alpha]})$ such that $\varphi_2(c_{\rho_{a_1}}) = c_0$ and $c_0 \preceq \varphi_2 \circ \varphi_1(a_2)$. By Claim~\ref{claim 1} again, there exists $\varphi_3 \in \Isom(T_\kappa^{[\alpha]})$ such that $\varphi_3(c_0) = c_0$ and $\varphi_3 \circ \varphi_2 \circ \varphi_1 (a_2) = c_r$. Let $\varphi \coloneqq \varphi_3 \circ \varphi_2 \circ \varphi_1$, so $\varphi(a_1) = c_0$ and $\varphi(a_2) = c_r$. By a similar argument, there exists $\psi \in \Isom(T_\kappa^{[\alpha]})$ such that $\psi(b_1) = c_0$ and $\psi(b_2) = c_r$. Then $\psi^{-1} \circ \varphi$ is the required isometry.
\end{proof}

\begin{lemma} \label{lem: very highly transitive}
	Let $\alpha$ be a countable ordinal. Then the following hold:
	\begin{enumerate}
		\item Let $x \in T_\kappa^{[\alpha]}$, let $I$ be a set of cardinality $\kappa$ and let $\{D_i : i \in I\}$ be the set of directions of $T_\kappa^{[\alpha]}$ at $x$. For any bijection $\sigma: I \rightarrow I$, there is an isometry $\psi \in \Isom(T_\kappa^{[\alpha]})$ such that $\psi(x) = x$ and $\psi(D_i) = D_{\sigma(i)}$ for each $i \in I$.
		\item Let $S \subseteq T_\kappa^{[\alpha]}$ be a closed subtree such that every point of $S$ has finite valence. For any isometry $\theta \in \Isom(S)$, there exists $\psi \in \Isom(T_\kappa^{[\alpha]})$ such that $\psi|_{S} = \theta$.
	\end{enumerate}
\end{lemma}
\begin{proof}
	We start by proving Item 1. Let $i,j \in I$ be distinct elements and let $x_i \in D_i, x_j \in D_j$ be such that $d(x,x_i) = d(x,x_j) = 1$. By Lemma~\ref{lem: 2-transitive}, there exists $\psi \in \Isom(T_\kappa^{[\alpha]})$ such that $\psi(x_i) = x_j$ and $\psi(x) = x$. Then $\psi(D_i) = D_j$, so $D_i$ and $D_j$ are isometric. For all $i,j \in I$, let $f_{i,j}: D_i \rightarrow D_j$ be an isometry. Now, if $\sigma: I \rightarrow I$ is a bijection, define $\psi: T_\kappa^{[\alpha]} \rightarrow T_\kappa^{[\alpha]}$ as follows. Let $\psi(x) \coloneqq x$ and, for all $i \in I$, let $\psi|_{D_i} \coloneqq f_{i,\sigma(i)}$. Then $\psi$ is the required isometry.
	
	Let us now prove Item 2. For each $x \in S$, let $\mathcal{D}_x$ be the set of directions of $T_\kappa^{[\alpha]}$ at $x$ which do not intersect $S$. The valence of $S$ at $\theta(x)$ is finite and equal to the valence of $S$ at $x$, and the valence of $T_\kappa^{[\alpha]}$ is $\kappa$ at every point of $T_\kappa^{[\alpha]}$. Thus there exists a bijection $\sigma_x : \mathcal{D}_x \rightarrow \mathcal{D}_{\theta(x)}$. We have shown that all directions at a point of $T_\kappa^{[\alpha]}$ are isometric. Since $T_\kappa^{[\alpha]}$ is homogeneous, it follows that there exists an isometry $f_D: D \rightarrow \sigma_x(D)$ for all $D \in \mathcal{D}_x$. Since $S$ is a closed subtree of $T_\kappa^{[\alpha]}$, every point in $T_\kappa^{[\alpha]} - S$ belongs to a direction $D$, where $D \in \mathcal{D}_x$ for some $x \in S$.
	We can therefore define $\psi:T_\kappa^{[\alpha]} \rightarrow T_\kappa^{[\alpha]}$ as follows. Let $\psi|_{S} \coloneqq \theta$ and, for each $x \in S$ and $D \in \mathcal{D}_x$, let $\psi|_{D} \coloneqq f_D$. Then $\psi$ is the required isometry.
\end{proof}

\section{Distinct isometry classes}

Let $\kappa \geq 3$ be a cardinal and let $\alpha \geq 1$ be a countable ordinal. 

\begin{definition}
    Let $a, b \in T_\kappa$ be such that $a \prec b$. The \textit{complexity of the pair $(a,b)$} is
    \[
        \comp(a,b) \coloneqq \rkCB(P_b \cap [\rho_a, \rho_b]).
    \]
\end{definition}

\begin{lemma} \label{lem: complexity of the limit}
    Let $(\beta_n)_{n \in \mathbb{N}}$ be a monotone increasing (possibly constant) sequence of countable successor ordinals such that $\beta_1 \geq 1$. Let $\beta \coloneqq \sup_{n \in \mathbb{N}} \beta_n$. Suppose $(a_n)_{n \in \mathbb{N}} \subseteq T_\kappa$ is a sequence such that $a_n \preceq a_{n+1}$ and $\comp(a_n, a_{n+1}) = \beta_n$ for each $n$, and there exists $a \in T_\kappa$ such that $(a_n)_{n \in \mathbb{N}}$ converges to $a$. Then $\comp(a_1,a) = \beta + 1$. Moreover $\comp(a) = \max\{\comp(a_1), \beta + 1\}$.
\end{lemma}
\begin{proof}
	Let $p \in P_a$ be such that $\rho_{a_1} \leq p < \rho_a$. Then $p \in [\rho_{a_n}, \rho_{a_{n+1}})$ for some $n \in \mathbb{N}$ and $P_a^{(\beta)} \cap [\rho_{a_n}, \rho_{a_{n+1}}) = P_{a_{n+1}}^{(\beta)} \cap [\rho_{a_n}, \rho_{a_{n+1}}) = \emptyset$, so $p \notin P_a^{(\beta)}$. On the other hand, if $\beta = \alpha +1$ for some $\alpha \geq 1$ then, for all sufficiently large $n$, there exists $p_n \in P_a^{(\alpha)} \cap [\rho_{a_n}, \rho_{a_{n+1}}]$ so $\rho_a = \lim_{n \rightarrow \infty} p_n \in P_a^{(\beta)}$. If $\beta$ is a limit ordinal then we can assume, up to passing to a subsequence, that $(\beta_n)_{n \in \mathbb{N}}$ is strictly increasing. Then, for all $n \geq 2$, there exists $p_n \in P_a^{(\beta_{n-1})} \cap [\rho_{a_n}, \rho_{a_{n+1}}]$, so $\rho_a = \lim_{n \rightarrow \infty} p_n \in \cap_{n \in \mathbb{N}} P_a^{(\beta_n)} = P_a^{(\beta)}$. Thus $P_a^{(\beta)} \cap [\rho_{a_1}, \rho_a] = \{\rho_a\}$ and $P_a^{(\beta + 1)} \cap [\rho_{a_1}, \rho_a] = \emptyset$.

    This shows that $\comp(a_1,a) = \beta + 1$. If $\comp(a_1) \leq \beta + 1$, then $P_a^{(\beta)} \cap (-\infty, \rho_{a_1})$ contains only isolated points. It follows that $P_a^{(\beta)}$ contains only isolated points, so $P_a^{(\beta+1)} = \emptyset$. Since $\comp(a_1, a) = \beta +1$, $P_a^{(\beta)} \neq \emptyset$, so $\comp(a) = \beta +1$. If $\comp(a_1) > \beta +1$ then, since $P_a^{(\alpha)} = P_{a_1}^{[\alpha]}$ for all $\alpha > \beta +1$, it follows that $\comp(a) = \comp(a_1)$.
\end{proof}

\begin{lemma} \label{lem: order preserving isom}
    Let $\varphi: \mathbb{R} \rightarrow T_\kappa^{[\alpha]}$ be a geodesic line. Then there exists an isometry $f \in \Isom(T_\kappa^{[\alpha]})$ such that $f(\varphi) = L_0$ and $f \circ \varphi((-\infty,0)) = \{c_r : r < 0\}$.
\end{lemma}
\begin{proof}
    Let $S \subseteq T_\kappa^{[\alpha]}$ be the convex hull of $L_0 \cup \varphi$. Then $S$ is a closed subtree whose points all have valence $\leq 4$, and there exists an isometry of $S$ which maps $\varphi$ onto $L_0$. By Lemma~\ref{lem: very highly transitive}, this implies that there is an isometry $f$ of $T_\kappa^{[\alpha]}$ which maps $\varphi$ onto $L_0$. Since $L_0$ is homogeneous, the same lemma implies that there is $g \in \Isom(T_\kappa^{[\alpha]})$ such that $g \circ f (\varphi) = L_0$ and $g \circ f \circ \varphi((-\infty,0)) = \{c_r : r < 0\}$.
\end{proof}

\begin{proposition} \label{lem: no isometry}
    Let $\kappa \geq 3$ be a cardinal and let $\alpha_1, \alpha_2 \geq 1$ be countable ordinals such that $\alpha_1 < \alpha_2$.
    \begin{itemize} 
    	\item The $\mathbb{R}$-tree $T_\kappa^{[\alpha_1]}$ is incomplete.
    	\item If $\psi: T_\kappa^{[\alpha_1]} \hookrightarrow T_\kappa^{[\alpha_2]}$ is an isometric embedding, then $\psi$ is not surjective.
    \end{itemize}
\end{proposition}
\begin{proof}
	By Lemma~\ref{lem: order preserving isom}, we can (and do) assume that $\psi$ restricts to the identity on $L_0$. This implies that $\psi$ is order preserving.

    \setcounter{claim}{0}
    \begin{claim}
     \label{claim: induction hyp}
    	Let $\beta$ be a successor ordinal such that $1 \leq \beta \leq \alpha_1$, let $x \in T_\kappa^{[\alpha_1]}$ and let $y \coloneqq \psi(x)$. Then, for all $r > 0$, there exists $a \in T_\kappa^{[\alpha_1]}$ such that:
    	\begin{itemize}
    		\item[i.] $x \prec a$;
    		\item[ii.] $\rho_a \leq \rho_x + r$;
    		\item[iii.] $\comp(x,a) = \beta$;
    		\item[iv.] $\comp(y, \psi(a)) \leq \beta$. 
    	\end{itemize}
    \end{claim}
    \begin{proof}
    	\renewcommand{\qedsymbol}{$\blacksquare$}
    	We proceed by transfinite induction on $\beta$. Suppose $\beta = 1$ and let $c \in C_\kappa - \{0\}$. If there exists $\varepsilon > 0$ such that $x((\rho_{x} - \varepsilon, \rho_{x})) = \{0\}$ then let $a': (-\infty, \rho_{x} + 1) \rightarrow C_\kappa$ be defined by 
    	\[
    	a'(t) =
    	\begin{cases}
    		x(t) \quad &\text{if } t < \rho_{x};\\
    		c &\text{otherwise.}
    	\end{cases}
    	\]
    	Otherwise, define $a': (-\infty, \rho_{x} + 1) \rightarrow C_\kappa$ by
    	\[ 
    	a'(t) =
    	\begin{cases}
    		x(t) \quad &\text{if } t < \rho_{x}; \\
    		0 &\text{otherwise.}
    	\end{cases}
    	\]
    	In either case $P_{a'} = P_{x} \cup \{\rho_{x}\}$. Moreover, if $a \in T_\kappa$ is such that $x \prec a \preceq a'$, then $P_{a} \cap [\rho_{x}, \rho_{a}] =\{\rho_{x}\}$, so $\comp(x, a) = 1$.
    	Let $r' > 0$ be such that $r' \leq r$ and, if $P_{\psi(a')} \cap (\rho_{y}, \rho_{\psi(a')}) \neq \emptyset$, then $\rho_{y} + r' \leq \min P_{\psi(a')} \cap (\rho_{y}, \rho_{\psi(a')})$.
    	Let $b \coloneqq \psi(a')|_{(-\infty, \rho_{y} + r')}$. Then $P_b \subseteq P_{y} \cup \{\rho_{y}\}$ so $\comp(y,b) \leq 1$.
    	Since $\psi\big(T_\kappa^{[\alpha_1]}\big)$ is connected, there exists $a \in T_\kappa^{[\alpha_1]}$ be such that $b = \psi(a)$. Then $x \prec a \preceq a'$ so $\comp(x, a) = 1$ and $\rho_{a} = \rho_{x} + d(x, a) = \rho_{x} + r' \leq \rho_{x} + r$.
    	
    	Now suppose $\beta = \gamma + 1> 1$ and the claim holds for all successor ordinals $1 \leq \beta' < \beta$. If $\gamma$ is a successor ordinal then let $\gamma_n \coloneqq \gamma$ for all $n \in \mathbb{N}$. If $\gamma$ is a limit ordinal, then let $(\gamma_n)_{n \in \mathbb{N}}$ be a strictly increasing sequence of successor ordinals such that $\sup_{n \in \mathbb{N}} \gamma_n = \gamma$.  We next define sequences $(a_n)_{n \in \mathbb{N}} \subseteq T_\kappa^{[\alpha_1]}$ and $(b_n)_{n \in \mathbb{N}} \subseteq T_\kappa^{[\alpha_2]}$ recursively such that, for all $n \in \mathbb{N}$:
    	\begin{enumerate}
    		\item $\psi(a_n) = b_n$;
    		\item $x \preceq a_n \prec a_{n+1}$ and $y \preceq b_n \prec b_{n+1}$;
    		\item $\comp(a_n,a_{n+1}) = \gamma_{n}$ and $\comp(b_n,b_{n+1}) \leq \gamma_{n}$;
    		\item $\rho_{a_n} \leq \rho_{x} + \sum_{i=1}^n r/2^i$ and $\rho_{b_n} \leq \rho_{y} + \sum_{i=1}^n r/2^i$.
    	\end{enumerate}
    	Let $a_1 \coloneqq x$ and $b_1 \coloneqq y$. Items 1,2 and 4 are satisfied. Let $n \in \mathbb{N}$ and suppose that $a_1, \dots, a_n \in T_\kappa^{[\alpha_1]}$ and $b_1, \dots, b_n \in T_\kappa^{[\alpha_2]}$ have been defined. Applying the induction hypothesis to $(\gamma_{n}, a_n, r/2^{n+1})$, there exists $a_{n+1} \in T_\kappa^{[\alpha_1]}$ such that 
    	\begin{itemize} 
    		\item[i.] $a_n \prec a_{n+1}$;
    		\item[ii.] $\rho_{a_{n+1}} \leq \rho_{a_n} + r/2^{n+1} \leq \rho_{x_1} + \sum_{i=1}^{n+1} r/2^i$;
    		\item[iii.] $\comp(a_n, a_{n+1}) = \gamma_{n}$;
    		\item[iv.] if $b_{n+1} \coloneqq \psi(a_{n+1})$, then $\comp(b_n, b_{n+1}) \leq \gamma_{n}$.
    	\end{itemize}
    	Thus Items 1-3 are satisfied. Moreover 
    	\[
    	\rho_{b_{n+1}} = \rho_{b_n} + d(b_n, b_{n+1}) \leq \rho_{y} + \sum_{i=1}^n \frac{r}{2^i} + d(a_n, a_{n+1}) \leq \rho_{y} + \sum_{i=1}^{n+1} \frac{r}{2^i},
    	\]
    	so Item~4 holds.
    	
    	It follows from Items 2 and 4 that the sequences $(a_n)_{n \in \mathbb{N}}$ and $(b_n)_{n \in \mathbb{N}}$ are strictly increasing and Cauchy. Let $a \coloneqq \lim_{n \rightarrow \infty} a_n \in T_\kappa$, $b \coloneqq \lim_{n \rightarrow \infty} b_n \in T_\kappa$. By Lemma~\ref{lem: complexity of the limit} and Item 3, $\comp(x,a) = \gamma + 1 = \beta$, $\comp(a) \leq \alpha_1$, $\comp(y,b) \leq \gamma + 1 < \alpha_2$ and $\comp(b) \leq \alpha_2$. Since $\psi$ is continuous, $\psi(a) = b$.
    	Moreover $x \prec a$ and $\rho_{a} = \lim_{n \rightarrow \infty} \rho_{a_n} \leq \rho_{x} + r$.
    \end{proof}

    Let $(\beta_n)_{n \in \mathbb{N}}$ be a monotone increasing (possibly constant) sequence of successor ordinals such that $\beta_n \geq 1$ for each $n$ and $\sup_{n \in \mathbb{N}} \beta_n = \alpha_1$.
    It follows from Claim~\ref{claim: induction hyp} (using the same argument as in the construction of $(a_n)_{n \in \mathbb{N}}$ and $(b_n)_{n \in \mathbb{N}}$ in the proof of that claim) that there exist sequences $(a_n)_{n \in \mathbb{N}} \subseteq T_\kappa^{[\alpha_1]}$ and $(b_n)_{n \in \mathbb{N}} \subseteq T_\kappa^{[\alpha_2]}$ such that $a_1 = b_1 = c_0$ and, for each $n \in \mathbb{N}$:
    \begin{enumerate}
            \item $\psi(a_n) = b_n$;
            \item $a_n \prec a_{n+1}$ and $b_n \prec b_{n+1}$;
            \item $\comp(a_n,a_{n+1}) = \beta_{n}$ and $\comp(b_n, b_{n+1}) \leq \beta_{n}$;
            \item $\rho_{a_n} \leq \rho_{a_1} + \sum_{i=1}^n r/2^i$ and $\rho_{b_n} \leq \rho_{b_1} + \sum_{i=1}^n r/2^i$.
    \end{enumerate}
    Items 2 and 4 imply that $(a_n)_{n \in \mathbb{N}}$ and $(b_{n \in \mathbb{N}})_{n \in \mathbb{N}}$ are Cauchy and Lemma~\ref{lem: complexity of the limit} implies that the complexity of $a \coloneqq \lim_{n \in \mathbb{N}} a_n \in T_\kappa$ is $\alpha_1 + 1$ and the complexity of $b \coloneqq \lim_{n \rightarrow \infty} b_n \in T_\kappa$ is $\leq \alpha_1 + 1 \leq \alpha_2$. Therefore $a \notin T_\kappa^{[\alpha_1]}$, so $T_\kappa^{[\alpha_1]}$ is incomplete, and $b = \lim_{n \rightarrow \infty} \psi(a_n) \in T_\kappa^{[\alpha_2]} - \psi(T_\kappa^{[\alpha_1]})$, so $\psi$ is not surjective.
\end{proof}

\bibliographystyle{alpha}
\bibliography{Biblio}

\bigskip
{\footnotesize
  \noindent
  {\textsc{University of Bristol, School of Mathematics, Bristol, UK}} \par\nopagebreak
  \texttt{penelope.azuelos@bristol.ac.uk}

\end{document}